\documentclass{article}
\usepackage{amsmath,amssymb,amsthm,fullpage}
\usepackage{arydshln,graphicx}

\renewcommand{\deg}{\operatorname{deg}}

\newtheorem{thm}{Theorem}[section]
\newtheorem{lem}[thm]{Lemma}
\newtheorem{cor}[thm]{Corollary}

\theoremstyle{definition}

\title{Graphs in which the Maxine heuristic produces a maximum independent set}
\author{Benjamin Lantz\\
\small Department of Mathematics\\
\small University of Rhode Island\\
\small Kingston, RI 02881\\
\small \tt benjamin\_lantz@uri.edu}

\begin{document}
\maketitle
\begin{abstract}
The residue of a graph is the number of zeros left after iteratively applying the Havel-Hakimi algorithm to its degree sequence. Favaron, Mah\'{e}o, and Sacl\'{e} showed that the residue is a lower bound on the independence number.  The Maxine heuristic reduces a graph to an independent set of size $M$. It has been shown that given a graph $G$, $M$ is bounded between the independence number and the residue of a graph for any application of the Maxine heuristic. We improve upon a forbidden subgraph classification of graphs such that $M$ is equal to the independence number given by Barrus and Molnar in 2015.
\end{abstract}

\section{Introduction} \label{sec: intro}
We will be considering simple graphs and we will let $N(v)$ represent the neighborhood of a vertex $v$ in a graph, and let $u\sim v$ mean that $u$ and $v$ are adjacent in the graph. For such a graph $G$ and subset of vertices $U$ in the graph, let $G[U]$ be the induced subgraph on the set $U$. For a set of graphs $\mathcal{S}$, a graph $G$ is said to be $\mathcal{S}$-free, if no graph in $\mathcal{S}$ appear as an induced subgraph in $G$.

Given a degree sequence $d=(d_1,d_2,\ldots, d_n)$, an iterative step in the Havel-Hakimi algorithm, developed independently by Havel \cite{Havel} and Hakimi \cite{Hakimi}, reduces $d$ to $d^1=(d_2-1,d_3-1,\ldots,d_{d_1+1}-1,d_{d_1+2}\ldots, d_n)$. After reordering the vertices to be non-increasing, the algorithm iterates until no positive entries are present. The algorithm arose to determine when a degree sequence is graphic: that is a list of integers $d$ is graphic if and only if the Havel Hakimi algorithm terminates in a list of zeros. The number of these zeros is said to be the residue of the degree sequence, and the residue of a graph $G$, denoted $R(G)$, is the residue of the degree sequence of $G$. The residue is of interest because of its connection to the independence number of a graph, $\alpha(G)$. In 1988, the conjecture-making computer program Graffiti \cite{Fajtlowicz} proposed the following theorem,  

\begin{thm}
\cite{FavaronMaheoSacle} For every graph $G$, $R(G)\leq \alpha(G)$.

\end{thm}

 This result was proven by Favaron et. al. in 1991 and improved upon by Griggs and Kleitman \cite{GriggsKleitman}, Triesch \cite{Triesch}, and Jelen \cite{Jelen} in the 1990's.    Determining the independence number is NP-hard, but since it takes only $O(E)$ steps to determine the residue where $E$ is the number of edges in a graph, it is of interest to know how well $R(G)$ approximates $\alpha(G)$ and when is the bound realized. 

To further illustrate the relationship between the residue and the independence number, we can consider the Maxine heuristic, which is the process of iteratively deleting vertices of maximum degree until an independent set of vertices is realized \cite{GriggsKleitman}. We will call $M$ the size of the independent set achieved by the Maxine heuristic and note that this is clearly a lower bound on the independence number. Note that the heuristic depends on our choice of deleted vertices and $M$ can vary accordingly. It was shown by Griggs and Kleitman \cite{GriggsKleitman} that 

\begin{thm}(\cite{GriggsKleitman})  If $M$ is the size of the independent set produced by any application of the Maxine heuristic for a graph $G$, then $R(G)\leq M\leq \alpha(G)$. 
\end{thm}

Thus if $R(G)=\alpha(G)$ for some $G$, then every application of the Maxine heuristic must achieve a maximum independent set. 

A vertex in a graph is said to have the Havel-Hakimi property if it is of maximum degree and its neighbors are of maximal degree, i.e. the deletion of said vertex corresponds to the reduction in the degree sequence by one step of the Havel-Hakimi algorithm. Not every graph has a vertex with this property, but every degree sequence has a realization that has such a vertex \cite{ChartrandLesniakZhang}. If at each step of the Maxine heuristic, a vertex with the Havel-Hakimi property is deleted, then $R(G)=M$. To find when $M=\alpha(G)$ we will consider graphs with certain conditions.  

A vertex $v$ in a graph $G$ is said to have maximum degree-independence conditions (or MDI conditions) if it is has maximum degree and is a part of every maximum independent set. Also we will say that a graph $G$ has maximum degree-independence conditions (or MDI conditions), if there exists a vertex $v\in V(G)$ that has MDI conditions.

 In 2016, Barrus and Molnar found that if a vertex $v$ in $G$ has MDI conditions, then $G$ must contain an induced subgraph of $C_4$ (the cycle on 4 vertices) containing $v$ or an induced subgraph of $P_5$ (the path on 5 vertices) with $v$ as the center vertex \cite{BarrusMolnar}. From this it can be quickly shown that 

\begin{thm}\label{BMthm} (\cite{BarrusMolnar}) 
The Maxine heuristic always produces a maximum independent set when applied to a $\{C_4,P_5\}$-free graph. 
\end{thm}

\section{Results} \label{sec: results}

We will work to strengthen Theorem \ref{BMthm} by examining the case where $v$ with MDI conditions is in an induced copy of $C_4$, since $C_4$ does not have MDI conditions itself. Since we will only strengthen the condition on $C_4$, we will assume that all graphs considered have no subgraph isomorphic to $P_5$ in which the center vertex has MDI conditions. We will call a graph $P_5*$-free  when referring to the condition that the center vertex must have MDI conditions, as we will not restrict the existence of an induced $P_5$ in general.  We will allude to the aforementioned MDI conditions as the maximum degree condition and independence condition separately. To start, we will prove a few lemmas to reduce our search of induced subgraphs needed to strengthen the $C_4$ condition.

\begin{lem}\label{UniqueIndLemma}
    If $v\in V(G)$ has MDI conditions and is a part of more than one maximum independent set, then there is an induced subgraph of $G$ in which $v$ also has MDI conditions and there is only one maximum independent set. 
\end{lem} 
    \begin{proof}
        Let $v$ belong to maximum independent sets $I_1,I_2,...,I_n$. Then we can consider the subgraph induced by deleting  $\bigcup_{i=2}^{n}I_i\setminus I_1$. The maximum degree condition is not violated since none of the deleted edges were adjacent to $v$, and there is exactly one maximum independent set in the induced subgraph. 
    \end{proof}
    \vspace{1 cm}
    
Because of Lemma \ref{UniqueIndLemma}, we will now only consider a graph $G$ with one maximum independent set $I$ including a vertex $v$ such that $v$ has MDI conditions.

\begin{lem}\label{OnlyN(v)} Let $x$ be a vertex such that $x\notin N(v)\cup I$ where $I$ is the lone independent set. Then $G\setminus\{x\}$ has MDI conditions as well. 
\end{lem}

\begin{proof}
Deleting $x$ does not change the degree of $v$ and thus the maximum degree condition is unaffected. Furthermore, since $x$ is not in $I$, the independent set is unaffected as well. Thus $v$ still has MDI conditions in $G\setminus\{x\}$. 
\end{proof}

\vspace{1 cm}

If $\alpha(G)=1$, $G$ must be a clique, and if there is only one maximum independent set, then $G$ must be an isolated vertex. Furthermore, if $\alpha(G)=2$ with maximum independent set $\{u,v\}$, then $N(v)=N(u)$ must form a clique and thus every element of $N(v)$ must have strictly larger degree than both $u$ and $v$. Since we require an element of the maximum independent set to have maximum degree, $N(v)$ must be empty and $G$ must be the graph of two isolated vertices. Hence, if $G$ has MDI conditions and $\alpha(G)\leq 2$, then every application of the Maxine heuristic vacuously produces a set of size $\alpha$. 

Thus we will now assume that the size of $I$ is 3 and that $I=\{u,v,w\}$ where $v$ is the vertex with MDI conditions and $I'=I\setminus\{v\}$. Note that if $x\in N(v)$ where $v$ has MDI conditions and $x$ is not adjacent to any other element in $I$, the maximum independent set,  then we have another maximum independent set $(I\setminus\{v\})\cup\{x\}$. From \ref{OnlyN(v)} we can then delete $x$ and retain conditions on $G$. Thus we only need to consider $N(v)\cup I$. We will then partition $N(v)$ into $Q_u$ and $Q_w$ as the vertices in $N(v)$ whose neighbors in $I'$ are only $u$ and $w$ respectively. We will call $Q=Q_u\cup Q_w$. Let $N$ be the set of vertices in $N(v)$ that are adjacent to both $u$ and $w$.  Since the independence number of $G$ must be 3 and $I$ is the unique independent set of size 3, we have that $Q_u$ and $Q_w$ must have independence number at most 1; hence $Q_u$ and $Q_w$ are cliques, since otherwise there would exist another independent set of size 3. Similarly, $N$ must have independence number at most 2. Then since $G$ must be $P_5^*$free, we have that $Q$ must form a clique as every vertex in $Q_u$ must dominate $Q_w$ and vice versa as otherwise there exists $q_u\in Q_u$ and $q_w\in Q_w$ non-adjacent; hence $\{u, q_u, v, q_w, w\}$ induce $P_5$ with $v$ as the center vertex.

\begin{thm}\label{Ind=3} Let $G$ have MDI conditions with $\alpha=3$. Then $G$ has at least one of the following induced subgraphs where $Q'$ is a subset of $Q$ and $N'$ a subset of $N$:
\begin{enumerate}
    \item $|Q'|=0$, $G[N']\cong\overline{C_n}$. 
    \item $|Q'|=1$, $G[N'\cup Q']\cong\overline{C_n}$.
    \item $|Q'|=2$, $G[N'\cup Q']\cong\overline{P_n}$ where the elements of $Q$ are the endpoints of $P_n$ in the complement.  
\end{enumerate}
\end{thm}

\begin{proof}

We will first consider the case where $|Q|=0$. First note that if $|N|=0$, then $N(v)$ is empty and $G$ is only the independent set and the result follows immediately. Thus we will assume that $N$ is non-empty.  We have that every vertex in $N$ has two non-neighbors in $N$ as $Q$ is empty and every vertex in $N$ is also adjacent to $u$, $v$, and $w$, otherwise $v$ would not have maximum degree as $N(v)=N\cup Q$. We can then arrange the non-neighbors into one or more disjoint cycle complements. Consider a smallest cycle complement, and label its vertices $x_0,\ldots, x_{m-1}$ where $x_i$ is non-adjacent to both $x_{i+1}$ and $x_{i-1}$ modulo $m$. If there exists an $x_i$ that does not dominate the rest of the cycle complement, then we have a smaller cycle complement which is a contradiction. Thus we have that $x_i$ dominates the rest of the cycle complement for every $i$ and thus we have $G[N']\cong \overline{C_m}$ where $N'$ is the vertex set of the cycle complement.

We will next consider the case where $|Q|=1$. We will call $q$ the lone vertex in $Q$. If $|N|=0$, then $q$ has larger degree than $v$, which is a contradiction so we will assume that $N$ is non-empty.  Note that every vertex in $N$ has to have at least 2 non-neighbors in $N\cup Q$ otherwise $v$ is not of maximum degree, as every vertex in $N$ is also adjacent to $u$, $v$, and $w$.  If $q$ dominates $N$ then $\deg(q)>\deg(v)$ which is a contradiction. Thus there exists a non-neighbor of $q$ in $N$; call it $x_0$, and call the other guaranteed non-neighbor of $x_0$, $x_1$. Similarly, $x_1$ is guaranteed to have another non-neighbor in $N\cup Q$ as $x_1\in N$ and must have at least two non-neighbors in $N\cup Q$. If this other non-neighbor is $q$ then we have that $\{q,x_0,x_1\}$ induce $\overline{C_3}$ and we are done. Thus we will assume that the other non-neighbor is in $N$, call it $x_2$. Inductively this creates a sequence of non-neighbors in $N$, $\{x_i\}$, as each $x_i$ must be adjacent to $q$ otherwise we are done as $\overline{C_{i+1}}$ is induced on $q\cup x_1\cup\cdots\cup x_i$. Furthermore each $x_i$ must be adjacent to $\{x_0,\ldots,x_{i-2}\}$ otherwise we have an induced copy of $\overline{C_n}$ in $N$ for some $n$.  Since we have a finite graph, this sequence must terminate at $x_m$ for some $m$, and thus we have that $x_m$ must be non-adjacent to either $q$ or some vertex in $\{x_0,\ldots,x_{m-2}\}$ giving the result.

Finally we will show the result if $|Q|\geq2$. We will proceed by induction on the size of $Q$. 
We will now consider the base case where $|Q|=2$, calling the 2 vertices $q_1,q_2$. Note that $q_1,q_2$ are adjacent as $Q$ forms a clique. Similar to the case $|Q|=1$, if $N$ is empty then $q_1$ has strictly larger degree than $v$ which is a contradiction. Thus we will assume that $N$ is non-empty. Each of the vertices has at least one non-neighbor in $N$; if they have the same non-neighbor then those three vertices induce the desired $\overline{P_3}$ and we are done. Thus we will assume that they have different non-neighbors, call them $x_1$ and $x_2$ respectively. If $x_1\nsim x_2$, then the four vertices induce the desired $\overline{P_4}$ and we are done, so assume that $x_1\sim x_2$. Each of these vertices has another non-neighbor in $N$; if they share a non-neighbor then the five vertices induce the desired $\overline{P_5}$, so assume that $x_1$ and $x_2$ have different non-neighbors call them $x_3$ and $x_4$ respectively. Inductively, we have that the pair of vertices $x_{2i},x_{2i+1}$ are the new non-neighbors of $x_{2i-2}$ and $x_{2i-1}$.  Note that $x_{2i}$ and $x_{2i+1}$  must be adjacent to $Q$ otherwise there is an induced complement of a cycle and we are done. Furthermore $x_{2i}$ must be adjacent to each $x_j$ for $j$ even and $x_{2i+1}$ must be adjacent to each $x_j$ for $j$ odd, otherwise we have an induced complement of a cycle in $N$. Then $x_{2i}$ must be adjacent to each $x_j$ for $j$ odd, and $x_{2i+1}$ must be adjacent to each $x_j$ for $j$ even, otherwise we have the desired induced complement of a path. We thus have that both $x_{2i}$ and $x_{2i+1}$ must have another non-neighbor in $N$. Since we have a finite graph, this process must terminate, yielding the result.

We will now show that if $|Q|>2$, $G$ has one of the desired induced subgraphs above. We will proceed by induction on $|Q|$, noting that the base case of $|Q|=2$ is done above. Assume the result is true for $|Q|<k$ and consider the case with $|Q|=k$. We will label the vertices of $Q$, $\{q_1,q_2,\ldots, q_k\}$. Each of these has a non-neighbor in $N$, call it $x_i$ for each $q_i$. Note that these are distinct otherwise we have an induced copy of $\overline{P_3}$ with 2 elements of $Q$ has endpoints in the complement. Furthermore $q_i\sim x_j$ for all $i\neq j$ as otherwise we have an induced $\overline{P_4}$. Then there exists another non-neighbor of $x_1$ in $N$, call it $y_1$. We have that 
\begin{itemize}
    \item $y_1\sim q_1$, otherwise $\{q_1,x_1,y_1\}$ induce $\overline{C_3}$. 
    
    \item $y_1\sim q_j$ for all $j>1$ otherwise $\{q_1,x_1,y_1,q_j\}$ induce $\overline{P_4}$. 
    
    \item $y_1\sim x_j$ for all $j>1$, otherwise $\{q_1,x_1,y_1,x_j,q_j\}$ induce $\overline{P_5}$. 
\end{itemize}

We then have that $y_1$ must have another non-neighbor in $N$, call it $y_2$. Inductively let $y_k$ be the other non-neighbor of $y_{k-1}$ where each $y_i$ for $1\leq i<k$ dominates all preceding vertices except $y_{i-1}$. Then we have that 
\begin{itemize}
    \item $y_k\sim q_1$, otherwise $\{q_1,x_1,y_1,\ldots,y_{k}\}$ induce $\overline{C_{k+2}}$.
    
    \item $y_k\sim q_j$ for all $j>1$, otherwise $\{q_1,x_1,y_1,\ldots,y_{k},q_j\}$ induce $\overline{P_{k+3}}$.
    
    \item $y_k\sim x_1$, otherwise $\{x_1,y_1,\ldots,y_{k}\}$ induce $\overline{C_{k+1}}$.
    
    \item $y_k\sim x_j$ for all $j>1$, otherwise $\{q_1,x_1,y_1,\ldots,y_{k},x_j,q_j\}$ induce $\overline{P_{k+4}}$
    
    \item $y_k\sim y_i$ for all $i<k$ otherwise inductively there is an induced complement of a cycle. 
\end{itemize}
Thus $y_k$ has another non-neighbor in $N$. Since our graph is finite, this process must terminate and the result holds.

\end{proof}

We will now extend the result to a graph with independence number greater than 3.

\begin{thm}\label{Ind=k} Let $G$ have MDI conditions with $\alpha=k$ such that $k>3$. Then the result from \ref{Ind=3} holds as well.
\end{thm}

\begin{proof}
We will assume the contrary, that there exists such a graph without the desired induced subgraphs and derive a contradiction. 

From Lemma \ref{UniqueIndLemma} we have that $G$ has one maximum independent set with $v$ a vertex with MDI conditions. First call $I$ the lone independent set, and $I'=I\setminus\{v\}$. Furthermore, we will use the notation that a set $A\subseteq N(v)$ induces a subgraph on $G_{ij}$, to mean $G[\{v,A,i,j\}]$ where $i,j$ are elements of $I'$. Then call $Q_{i}\subseteq N(v)$ the vertices that are adjacent to exactly $i$ members of $I'$. Then $\{Q_i\}_{i=1}^{k-1}$ partition $N(v)$, using Lemma \ref{OnlyN(v)}. Also note, that in order for $G$ to have MDI conditions, every vertex in $Q_i$ must have $i$ non-neighbors in $N(v)$, otherwise $v$ would not have maximum degree. We will first show that $Q_{k-1}$ must be empty.

Let $q\in Q_{k-1}$. We will show that $q$ must have at most one non-neighbor in $N(v)\setminus Q_{k-1}$. Suppose that $q$ has two such non-neighbors; $x$ and $y$.

First suppose $x\sim y$. If $x$ and $y$ have distinct neighbors in $I'$, call them $u$ and $w$ respectively, then $\{q,x,y\}$ induce $\overline{P_3}$ in $G_{u,w}$. Otherwise, without loss of generality, $(N(x)\cap I')\subseteq (N(y)\cap I')$, and we must have that $N(x)\cap I'$ is non-empty, so it contains an element $u$, and $(N(y)\cap I')^c$ is non-empty as $y\notin Q_{k-1}$, and thus $w\in (N(y)\cap I')^c $. We then have that, again, $\{q,x,y\}$ induce $\overline{P_3}$ in $G_{u,w}$.

Then suppose that $x\nsim y$. We must have that $x$ and $y$ do not have any distinct neighbors in $I'$, say $a$ and $b$, as otherwise $\{x,v,y,a,b\}$ would induce $P_5$. Then $x$ and $y$ share a neighbor in $I'$, call it $u$ and note that both $x,y$ cannot belong to $Q_1$, as $Q_1$ forms a clique. Thus, without loss of generality, we can say that $y$ has another neighbor, $w$, in $I'$, and thus $\{q,x,y\}$ induce $\overline{C_3}$ in $G_{u,w}$.

We thus have that, for each $q\in Q_{k-1}$, $q$ must have at most one non-neighbor in $N(v)\setminus Q_{k-1}$, and thus must have at least 2 non-neighbors in $Q_{k-1}$. As in the proof of the $\alpha=3$ case, we can arrange a smallest cycle complement of non-neighbors and thus we have an induced $\overline{C_n}$ in $G_{u,w}$ where $u,w$ are any two members of $I'$. This is a contradiction, and thus $Q_{k-1}$ must be empty.

We will then proceed by induction to show that $Q_i$ is empty for $3\leq i\leq k-1$. We will assume that $Q_{i}$ is empty for all $i>\ell$, and we will show that $Q_{\ell}$ is empty as well.

let $q\in Q_{\ell}$, and assume that $q$ has two non-neighbors in $N(v)\setminus Q_{\ell}$, call them $x$ and $y$. If any pair of $\{x,y,q\}$ have distinct neighbors in $I'$, then we have an induced $P_5$, as seen above. Thus we must have that without loss of generality, $(N(x)\cap I')\subseteq (N(y)\cap I')$, and since $q$ has the most neighbors in $I'$, $(N(y)\cap I')\subseteq (N(q)\cap I')$. Then we argue, in the same way as in the base case of $Q_{k-1}$, that $q$ can only have at most one non-neighbor in $N(v)\setminus Q_{\ell}$. Thus, $q$ has at least $\ell-1$ non-neighbors in $Q_{\ell}$ as $Q_i$ is empty for all $i>\ell$, and as above this means that we have an induced $\overline{C_n}$, a contradiction. Thus we have that $Q_{\ell}$ must be empty. Hence by induction we have that $Q_i$ is empty for all $i>2$.

Note that $N(v)$ must be non-empty, as we cannot have an edgeless graph, and $Q_2$ cannot be empty as $Q_1$ forms a clique, and each element of $Q_1$ must have at least one non-neighbor in $N(v)$. Then let $q\in Q_2$. If $q$ has two non-neighbors in $Q_1$, adjacent to $u$ and $w$ respectively in $I'$, then $q$ must also be adjacent to $u,w$, otherwise we have an induced $P_5$. Thus the three vertices induce $\overline{P_2}$ in $G_{u,w}$. Then assume that $q$ has exactly one non-neighbor in $Q_1$, call it $x$ and a non-neighbor in $Q_2$, call it $y$. We must have that $q,y$ share the same neighbors in $I'$, otherwise we have an induced $P_5$, and thus the neighbor of $x$ in $I'$ is shared by both $q,y$. We then have that if $x\nsim y$, we have that $\{x,y,q\}$ induce $\overline{C_3}$. We will thus assume that $x\sim y$.

Then if all $q\in Q_2$ have 2 non-neighbors in $Q_2$, we must have an induced copy of $\overline{C_n}$ in $Q_2$. Suppose then that there are 2 vertices in $Q_2$, $q,q'$ that have a non-neighbor in $Q_1$, and choose these vertices such that the distance between them in $Q_2^c$ is as small as possible. Note that there must exist a chain of vertices in $Q_2$ such that $q\nsim q_1\nsim q_2\nsim\cdots\nsim q'$, such that $q_i$ does not have a non-neighbor in $Q_1$. Furthermore, $q,q',q_i$ must share the same neighbors in $I'$, otherwise we have an induced copy of $P_5$. If $q,q'$ have the same non-neighbor in $Q_1$, call it $x$, then $\{x,q,q',q_1,\ldots\}$ induce $\overline{C_n}$. If $q,q'$ have different non-neighbors, $x$ and $x'$ in $Q_1$, then $\{x,x',q,q',q_1,\ldots\}$ induce $\overline{P_n}$. This is a contradiction, and thus for every graph $G$ with conditions and $\alpha=k>3$, we have the result.  
\end{proof}

For ease, we will call the families of induced subgraphs in \ref{Ind=3} $\mathcal{F}$. We wanted to improve the $C_4$ condition introduced by Barrus and Molnar as $C_4$ itself was not MDI. By construction, each graph in $\mathcal{F}$ is itself MDI alongside $P_5$. We then have the immediate corollary, 

\begin{cor}
The Maxine heuristic always produces a maximum independent set when applied to a $\{\mathcal{F},P_5\}$-free graph.
\end{cor}

\section{Open Questions} \label{sec: open questions}

Barrus and Molnar used their results to show that if a graph is $\{P_5, 4-\text{pan}, K_{2,3}, K^+_{2,3}, \text{kite}, 2P_3, P_3+ K_3, \text{stool}, \text{co-domino}\}$-free, then $R(G)=\alpha(G)$.\cite{BarrusMolnar} It can be expected that this class of graphs can be expanded with the strengthened conditions shown in this paper. We pose the following open questions/problems:
\begin{itemize}
    \item Can we fully classify the graphs in which the Maxine heuristic produces a maximum independent set. 
    
    \item What other conditions, other than forbidding MDI conditions, can be considered to guarantee that the Maxine heuristic produces a maximum independent set?
    
    \item Can we fully classify the graphs in which the Maxine heuristic produces a graph with an independent set the same size as the residue? Note that graphs with the Havel-Hakimi property introduced in \cite{BarrusMolnar} are a subset of these graphs. 

    \item Can we fully classify the graphs in which the residue equals the independence number?

\end{itemize}

\end{document}